\documentclass[reqno,12pt,a4paper]{amsart}

\usepackage{amsmath,amsfonts,amsthm,amssymb,amsxtra}
\usepackage{mathrsfs}
\usepackage{amssymb}
\usepackage{amsfonts}
\usepackage{latexsym}
\usepackage{amsthm}

\usepackage{graphicx}

\begin{document}


\renewcommand{\theequation}{\arabic{section}.\arabic{equation}}
\theoremstyle{plain}
\newtheorem{theorem}{\bf Theorem}[section]
\newtheorem{lemma}[theorem]{\bf Lemma}
\newtheorem{corollary}[theorem]{\bf Corollary}
\newtheorem{proposition}[theorem]{\bf Proposition}
\newtheorem{definition}[theorem]{\bf Definition}
\newtheorem{remark}[theorem]{\it Remark}

\def\a{\alpha}  \def\cA{{\mathcal A}}     \def\bA{{\bf A}}  \def\mA{{\mathscr A}}
\def\b{\beta}   \def\cB{{\mathcal B}}     \def\bB{{\bf B}}  \def\mB{{\mathscr B}}
\def\g{\gamma}  \def\cC{{\mathcal C}}     \def\bC{{\bf C}}  \def\mC{{\mathscr C}}
\def\G{\Gamma}  \def\cD{{\mathcal D}}     \def\bD{{\bf D}}  \def\mD{{\mathscr D}}
\def\d{\delta}  \def\cE{{\mathcal E}}     \def\bE{{\bf E}}  \def\mE{{\mathscr E}}
\def\D{\Delta}  \def\cF{{\mathcal F}}     \def\bF{{\bf F}}  \def\mF{{\mathscr F}}
\def\c{\chi}    \def\cG{{\mathcal G}}     \def\bG{{\bf G}}  \def\mG{{\mathscr G}}
\def\z{\zeta}   \def\cH{{\mathcal H}}     \def\bH{{\bf H}}  \def\mH{{\mathscr H}}
\def\e{\eta}    \def\cI{{\mathcal I}}     \def\bI{{\bf I}}  \def\mI{{\mathscr I}}
\def\p{\psi}    \def\cJ{{\mathcal J}}     \def\bJ{{\bf J}}  \def\mJ{{\mathscr J}}
\def\vT{\Theta} \def\cK{{\mathcal K}}     \def\bK{{\bf K}}  \def\mK{{\mathscr K}}
\def\k{\kappa}  \def\cL{{\mathcal L}}     \def\bL{{\bf L}}  \def\mL{{\mathscr L}}
\def\l{\lambda} \def\cM{{\mathcal M}}     \def\bM{{\bf M}}  \def\mM{{\mathscr M}}
\def\L{\Lambda} \def\cN{{\mathcal N}}     \def\bN{{\bf N}}  \def\mN{{\mathscr N}}
\def\m{\mu}     \def\cO{{\mathcal O}}     \def\bO{{\bf O}}  \def\mO{{\mathscr O}}
\def\n{\nu}     \def\cP{{\mathcal P}}     \def\bP{{\bf P}}  \def\mP{{\mathscr P}}
\def\r{\rho}    \def\cQ{{\mathcal Q}}     \def\bQ{{\bf Q}}  \def\mQ{{\mathscr Q}}
\def\s{\sigma}  \def\cR{{\mathcal R}}     \def\bR{{\bf R}}  \def\mR{{\mathscr R}}
                \def\cS{{\mathcal S}}     \def\bS{{\bf S}}  \def\mS{{\mathscr S}}
\def\t{\tau}    \def\cT{{\mathcal T}}     \def\bT{{\bf T}}  \def\mT{{\mathscr T}}
\def\f{\phi}    \def\cU{{\mathcal U}}     \def\bU{{\bf U}}  \def\mU{{\mathscr U}}
\def\F{\Phi}    \def\cV{{\mathcal V}}     \def\bV{{\bf V}}  \def\mV{{\mathscr V}}
\def\P{\Psi}    \def\cW{{\mathcal W}}     \def\bW{{\bf W}}  \def\mW{{\mathscr W}}
\def\o{\omega}  \def\cX{{\mathcal X}}     \def\bX{{\bf X}}  \def\mX{{\mathscr X}}
\def\x{\xi}     \def\cY{{\mathcal Y}}     \def\bY{{\bf Y}}  \def\mY{{\mathscr Y}}
\def\X{\Xi}     \def\cZ{{\mathcal Z}}     \def\bZ{{\bf Z}}  \def\mZ{{\mathscr Z}}
\def\O{\Omega}

\newcommand{\gA}{\mathfrak{A}}
\newcommand{\gB}{\mathfrak{B}}
\newcommand{\gC}{\mathfrak{C}}
\newcommand{\gD}{\mathfrak{D}}
\newcommand{\gE}{\mathfrak{E}}
\newcommand{\gF}{\mathfrak{F}}
\newcommand{\gG}{\mathfrak{G}}
\newcommand{\gH}{\mathfrak{H}}
\newcommand{\gI}{\mathfrak{I}}
\newcommand{\gJ}{\mathfrak{J}}
\newcommand{\gK}{\mathfrak{K}}
\newcommand{\gL}{\mathfrak{L}}
\newcommand{\gM}{\mathfrak{M}}
\newcommand{\gN}{\mathfrak{N}}
\newcommand{\gO}{\mathfrak{O}}
\newcommand{\gP}{\mathfrak{P}}
\newcommand{\gQ}{\mathfrak{Q}}
\newcommand{\gR}{\mathfrak{R}}
\newcommand{\gS}{\mathfrak{S}}
\newcommand{\gT}{\mathfrak{T}}
\newcommand{\gU}{\mathfrak{U}}
\newcommand{\gV}{\mathfrak{V}}
\newcommand{\gW}{\mathfrak{W}}
\newcommand{\gX}{\mathfrak{X}}
\newcommand{\gY}{\mathfrak{Y}}
\newcommand{\gZ}{\mathfrak{Z}}

\def\ve{\varepsilon} \def\vt{\vartheta} \def\vp{\varphi}  \def\vk{\varkappa}

\def\Z{{\mathbb Z}} \def\R{{\mathbb R}} \def\C{{\mathbb C}}  \def\K{{\mathbb K}}
\def\T{{\mathbb T}} \def\N{{\mathbb N}} \def\dD{{\mathbb D}} \def\S{{\mathbb S}}
\def\B{{\mathbb B}}


\def\la{\leftarrow}              \def\ra{\rightarrow}     \def\Ra{\Rightarrow}
\def\ua{\uparrow}                \def\da{\downarrow}
\def\lra{\leftrightarrow}        \def\Lra{\Leftrightarrow}
\newcommand{\abs}[1]{\lvert#1\rvert}
\newcommand{\br}[1]{\left(#1\right)}

\def\lan{\langle} \def\ran{\rangle}


\def\lt{\biggl}                  \def\rt{\biggr}
\def\ol{\overline}               \def\wt{\widetilde}
\def\no{\noindent}


\let\ge\geqslant                 \let\le\leqslant
\def\lan{\langle}                \def\ran{\rangle}
\def\/{\over}                    \def\iy{\infty}
\def\sm{\setminus}               \def\es{\emptyset}
\def\ss{\subset}                 \def\ts{\times}
\def\pa{\partial}                \def\os{\oplus}
\def\om{\ominus}                 \def\ev{\equiv}
\def\iint{\int\!\!\!\int}        \def\iintt{\mathop{\int\!\!\int\!\!\dots\!\!\int}\limits}
\def\el2{\ell^{\,2}}             \def\1{1\!\!1}
\def\sh{\sharp}
\def\wh{\widehat}
\def\bs{\backslash}
\def\na{\nabla}

\def\sh{\mathop{\mathrm{sh}}\nolimits}
\def\all{\mathop{\mathrm{all}}\nolimits}
\def\Area{\mathop{\mathrm{Area}}\nolimits}
\def\arg{\mathop{\mathrm{arg}}\nolimits}
\def\const{\mathop{\mathrm{const}}\nolimits}
\def\det{\mathop{\mathrm{det}}\nolimits}
\def\diag{\mathop{\mathrm{diag}}\nolimits}
\def\diam{\mathop{\mathrm{diam}}\nolimits}
\def\dim{\mathop{\mathrm{dim}}\nolimits}
\def\dist{\mathop{\mathrm{dist}}\nolimits}
\def\Im{\mathop{\mathrm{Im}}\nolimits}
\def\Iso{\mathop{\mathrm{Iso}}\nolimits}
\def\Ker{\mathop{\mathrm{Ker}}\nolimits}
\def\Lip{\mathop{\mathrm{Lip}}\nolimits}
\def\rank{\mathop{\mathrm{rank}}\limits}
\def\Ran{\mathop{\mathrm{Ran}}\nolimits}
\def\Re{\mathop{\mathrm{Re}}\nolimits}
\def\Res{\mathop{\mathrm{Res}}\nolimits}
\def\res{\mathop{\mathrm{res}}\limits}
\def\sign{\mathop{\mathrm{sign}}\nolimits}
\def\span{\mathop{\mathrm{span}}\nolimits}
\def\supp{\mathop{\mathrm{supp}}\nolimits}
\def\Tr{\mathop{\mathrm{Tr}}\nolimits}
\def\BBox{\hspace{1mm}\vrule height6pt width5.5pt depth0pt \hspace{6pt}}
\def\where{\mathop{\mathrm{where}}\nolimits}
\def\as{\mathop{\mathrm{as}}\nolimits}


\newcommand\nh[2]{\widehat{#1}\vphantom{#1}^{(#2)}}
\def\dia{\diamond}

\def\Oplus{\bigoplus\nolimits}


\def\qqq{\qquad}
\def\qq{\quad}
\let\ge\geqslant
\let\le\leqslant
\let\geq\geqslant
\let\leq\leqslant
\newcommand{\ca}{\begin{cases}}
\newcommand{\ac}{\end{cases}}
\newcommand{\ma}{\begin{pmatrix}}
\newcommand{\am}{\end{pmatrix}}
\renewcommand{\[}{\begin{equation}}
\renewcommand{\]}{\end{equation}}
\def\eq{\begin{equation}}
\def\qe{\end{equation}}
\def\[{\begin{equation}}
\def\bu{\bullet}

\title[Eigenvalues]{ Eigenvalues  of the  discrete  Schr\"odinger operator in the large  coupling  constant  limit}
\author{Siyu Gao}
\address{Department of Mathematics and Statistics, UNCC, 9201 University City Blvd., 
Charlotte, NC 28223, USA}


\begin{abstract}
Let $(\lambda_-,\lambda_+)$ be a spectral gap of a periodic Schr\"odinger operator $A$ on the lattice ${\mathbb Z}^d$. Consider the operator $A(\alpha)=A-\alpha V$
where $V$ is a decaying positive potential on ${\mathbb Z}^d$. We study the asymptotic behavior of the number of eigenvalues of $A(t)$ passing through a point $\lambda\in (\lambda_-,\lambda_+)$ as $t$ grows from $0$ to $\alpha$.
\end{abstract}

\maketitle

\section {Introduction and main results}
\setcounter{equation}{0}

We study operator $A=-\Delta+f$ defined on $\ell^2({\mathbb Z}^d)$ by $$[A u]_n=-\sum_{|m-n|=1}u_m+2d u_n+f(n)u_n,$$ where $f:\mathbb Z^d\to\mathbb R$ is a bounded function on the lattice $\mathbb Z^d$. Let $V:{\mathbb Z}^d\to [0,\infty)$ be a real potential having the property 
\begin{equation}\label{Vasympt}
    V(n)\sim
        \Psi(\theta)|n|^{-p},\qquad \text{ as  } |n|\to \infty,
\end{equation}
where $\Psi$ is a continuous  function on the  unit sphere $\{x\in {\mathbb R}^d:\,\,|x|=1\}$ and $\theta=n/|n|$. Then the operator of multiplication by the function $V$ is compact. 

Define the operator $A(\alpha)$ to be
\[
A(\alpha)=A-\alpha V,\qquad \alpha>0.
\]
Assume that the spectrum of $A$ has a bounded gap $\Lambda=(\lambda_-,\lambda_+)$, that is, $\Lambda\cap\sigma(A)=\emptyset$ and $\lambda_\pm\in\sigma(A)$. By Weyl's theorem, the spectrum of operator $A(\alpha)$ is discrete in $\Lambda$, so it consists of isolated eigenvalues of finite multiplicity. These eigenvalues move from the right to the left with the growth of $\alpha$ and depend on $\alpha$ continuously. For a point $\lambda\in\Lambda$, we define $N(\lambda,\alpha)$ to be the number of eigenvalues $N(\lambda,\alpha)$ of $A(t)$ that pass through $\lambda$ as $t$ increases from $0$ to $\alpha$. 

Also, for any self-adjoint operator $T=T^*$, we set $N(T,\lambda)=\text{rank}(E_T(-\infty,\lambda))$, where $E_T$ is the spectral measure of $T$.

In order to state our main theorem,  we  need to introduce the so-called density of states for the unperturbed operator $A$.  For that purpose,  for  any  domain $\Omega\subset {\mathbb R}^d$, we consider the operator $A_{\Omega}$  defined  on  $\ell^2(\Omega\cap\mathbb Z^d)$ by
\[
\bigl[ A_{\Omega}u\bigr]_n=-\sum_{|m-n|=1,\,  m\in \Omega\cap\mathbb Z^d}u_m+ 2d u_n+f(n)u_n,\qquad \forall n \in \Omega\cap\mathbb Z^d.
\]
Denote  by  $N(A_\Omega,\lambda)$ the number of eigenvalues of operator $A_\Omega$ that are strictly smaller than $\lambda\in\mathbb R$.  For $\beta>0$, we  define    ${\beta\Omega}$ to be the  domain obtained  from $\Omega$ by  scaling:
\[
\beta\Omega=\{x\in {\mathbb R}^d:\,\, x=\beta y\quad \text{for some }\,\, y\in \Omega\}.
\]
Then we consider the number $N(A_{\beta\Omega},\lambda)$ and study its behavior for large values of $\beta$.

{\bf Condition} Assume that the following limit exists for each $\lambda\in\mathbb R$,
\begin{eqnarray*}
    \rho(\lambda)=\lim_{\beta\to\infty}\frac{N(A_{\beta\Omega},\lambda)}{\beta^d\text{vol}\Omega}.
\end{eqnarray*} This limit is called the density of states.

 It is  known that, $\rho$ is an increasing function of $\lambda$ having the property
\begin{eqnarray*}
\rho(\lambda)=\begin{cases}0,\quad \text{for}\quad \lambda<-\|f\|_\infty;\\
1,\quad \text{for}\quad \lambda>4d+\|f\|_\infty.
\end{cases}
\end{eqnarray*} 
Our main result  establishes  the asymptotics  of $N(\lambda,\alpha)$ as $\alpha\to\infty$.

\begin{theorem}
\label{main theorem}
Let the potential $V$  obey the condition \eqref{Vasympt}. Assume that $\lambda\in\Lambda$. Let $N(\lambda,\alpha)$ be the  number of eigenvalues of $A(t)$ passing through $\lambda$ as $t$ increases from $0$ to $\alpha$.  Then
\begin{eqnarray}
N(\lambda,\alpha)\sim \alpha^{d/p} \int_{{\mathbb R}^d}\Bigl(\rho(\lambda+\Psi(\theta)|x|^{-p})-\rho(\lambda)\Bigr)dx,\qquad  \text{as}\quad \alpha\to\infty.
\end{eqnarray}
\end{theorem}

Operators with periodic potentials $f$ play a very important role in the so-called one-electron approximation model used in the quantum theory of crystals. According to this theory, the ``allowed" energies of an electron moving in a crystal lie in $\sigma(A)$, the spectrum of $A$, consisting of bands.
For a typical insulator, there is at least one gap in $\sigma(A)$. The $Al_2O_3$-crystal is one of examples of periodic media with this property. The theory also says that a photon whose energy is smaller than the length of the gap can not be absorbed by the crystal, because all states in the first band of the spectrum are already filled (see also \cite{DH}). That is the reason why Corundum (i.e. $Al_2O_3$) is colorless. However, replacing some of the $Al^{3+}$- ions by either $Cr^{3+}$ or $Ti^{3+}$-ions, one
obtains Ruby or Sapphire which absorb green and yellow colors. That is the reason why Ruby or Sapphire look red and blue respectively. Due to the fact that the crystal is not pure, there are additional isolated energy levels in spectral gaps of the pure crystal. Since the distances from these levels to the edges of the gap are smaller than the length of the gap,
it is easier for the light to be absorbed by a crystal with impurities. In our model, the function $V$ plays the role of the impurity potential.

While such problems on the lattice are considered for the first time, similar problems involving continuous operators on ${\mathbb R}^d$ have been studied before. For instance the main result of \cite{Hempel2}
says that, if $V$ decays sufficiently fast at infinity, then the number $N(\lambda,\alpha)$ of eigenvalues of the continuous Schrodinger operator $H(t)=-\Delta+f-tV$ passing through a regular point $\l\notin \sigma(H)$ as $t$ increases from $0$ to $\alpha$ satisfies
\begin{eqnarray}
\label{classical}
N(\lambda,\alpha)\sim (2\pi)^{-d} \omega_d \alpha^{d/2}\int_{{\mathbb R}^d} V^{d/2}dx,\qquad \text{as}\quad \alpha\to\infty.
\end{eqnarray}
Here $\omega_d$ is the volume of the unit ball in ${\mathbb R}^d$. An interesting short proof of \eqref{classical} was given by M. Birman in \cite{Birman}.
The author showed in \cite{Birman} that the asymptotics of $N(\lambda,\alpha)$ does not depend on the point $\lambda$ and the potential $f$.
Consequently, one can take $\lambda=-1$ and set $f=0$. After that, one can use previously known results.

The articles
\cite{DH} and \cite{Klaus2} are probably the two earliest publications in which the authors discussed the eigenvalues of $H(\alpha)$ in a bounded spectral gap of $H$.
Perturbations in \cite{DH} were not necessarily sign-definite and the question was whether the quantity $ N(\lambda,\alpha)$ is positive for some $\alpha>0.$ Similar questions
were studied in \cite{GS}. While the papers that we mentioned answered some important questions, they did not contain asymptotic formulas for the quantity
$N(\l,\a)$, which appeared later in \cite{ADH} and \cite{Hempel2}.
Among the other results of \cite{ADH}, this work of S. Alama, P. Deift, and R. Hempel contains an asymptotic formula for $N(\l,\a)$ in the case $V(x)\sim -c|x|^{-p}$ as $|x|\to\infty$, with $p, c>0$. In the corresponding statement of \cite{ADH}, the potential $V\leq0$ is nonpositive, and the eigenvalues of $H(\a)$ move from the left to the right.
The methods that are close to the ones of the paper \cite{ADH} were used by
R. Hempel in \cite{Hempel2} and \cite{Hempel}. In particular, he proved \eqref{classical} for $\lambda$ that belongs to a finite gap (see \cite{Hempel2}).

An interesting phenomenon remotely related to our study was discovered by the authors of \cite{Geszt} (by F. Gesztesy et al.).
It turns out that, if $V\geq 0$ is compactly supported and $d=1$, then the eigenvalues of $H(\alpha)$
move slowly near some very specific points of the gap and then move faster once they pass these points. This type of behavior of eigenvalues is called the ``cascading" in \cite{Geszt}.

\section{Birman-Schwinger Principle}

Here we describe the way to reduce the study of eigenvalues of operator $A(\alpha)$ in $\Lambda$ to the study of the spectrum of the compact operator
\begin{eqnarray}
\label{X}
    X(\lambda)=\sqrt{V}(A-\lambda I)^{-1}\sqrt{V}.
\end{eqnarray} Namely, for a self-adjoint operator $T=T^*$, we define $n_+(s,T)$ to be the number of eigenvalues of $T$ that are greater than $s>0$.

\begin{lemma}
    Assume that $\lambda\in\Lambda$. Let $X(\lambda)$ be defined in \eqref{X} and $N(\lambda,\alpha)$ be the  number of eigenvalues of $A(t)$ passing through $\lambda$ as $t$ increases from $0$ to $\alpha$. Then
    \begin{eqnarray}
    \label{Birman-Schwinger Principle}    N(\lambda,\alpha)=n_+(s,X(\lambda)), \qquad s=\alpha^{-1}.    \end{eqnarray}
\end{lemma}

The proof of the lemma can be found in \cite{Birman}. This lemma is called the Birman-Schwinger Principle.

\section{Splitting principle}

In this  section, we  justify the so-called  splitting principle described below.
Let $0<\varepsilon_1<\varepsilon_2<\infty$.
 We decompose  ${\mathbb R}^d$  into  three regions:
\begin{eqnarray*}
\Omega_1(\alpha)&=&\{x\in {\mathbb R}^d: \,\, |x|<\varepsilon_1\alpha^{1/p}\},
\\
\Omega_2(\alpha)&=&\{x\in\mathbb R^d:\varepsilon_1\alpha^{1/p}\le|x|\le\varepsilon_2\alpha^{1/p}\},
\\
\Omega_3(\alpha)&=&\{x\in {\mathbb R}^d: \,\, |x|>\varepsilon_2\alpha^{1/p}\}.
\end{eqnarray*} 
We choose $\varepsilon_2$ so that $\varepsilon_2>(\|\Psi\|_\infty/|\lambda_+-\lambda|+\|\Psi\|_\infty/|\lambda_--\lambda|)^{1/p}$. For each region $\Omega_k(\alpha)$, we consider the operator  $A_{\Omega_k(\alpha)}$, where $k=1,2,3$. For $\lambda\in\mathbb R$, we define
\begin{eqnarray}
\label{N_k}    N_k(\lambda,\alpha)=\text{rank E}_{A_{\Omega_k}(\alpha)-\alpha V}(-\infty,\lambda)-\text{rank E}_{A_{\Omega_k}(\alpha)}(-\infty,\lambda), \qquad k=1,2.
\end{eqnarray}

\begin{proposition}
    Let $\varepsilon_2>(\|\Psi\|_\infty/|\lambda_+-\lambda|+\|\Psi\|_\infty/|\lambda_--\lambda|)^{1/p}$ and  $N_k(\lambda,\alpha)$ be defined above for $k=1,2$. Then \[
N(\lambda,\alpha)=N_1(\lambda,\alpha)+N_2(\lambda,\alpha)+O(\alpha^{(d-1)/p}),\qquad\text{as}\quad \alpha\to\infty.
\]
\end{proposition}
\begin{proof}
    First, observe that since $\mathbb Z^d$ is the disjoint union of the sets $\Omega_k(\alpha)\cap\mathbb Z^d, k=1,2,3$, we have
    \begin{eqnarray*}
        l^2(\mathbb Z^d)=l^2(\Omega_1(\alpha)\cap\mathbb Z^d)\oplus l^2(\Omega_2(\alpha)\cap\mathbb Z^d)\oplus l^2(\Omega_3(\alpha)\cap\mathbb Z^d).
    \end{eqnarray*} This decomposition allows one to consider the operator
    \begin{eqnarray*}
        B_0=A_{\Omega_1(\alpha)}\oplus A_{\Omega_2(\alpha)}\oplus A_{\Omega_3(\alpha)}.
    \end{eqnarray*} This operator might have eigenvalues inside $(\lambda_-,\lambda_+)$. However, the operator $$B=B_0-(\lambda_+-\lambda_-)E_{B_0}(\lambda_-,\lambda_+)$$ does not have this property: the spectrum of $B$ in $(\lambda_-,\lambda_+)$ is empty. We set now $B(\alpha)=B-\alpha V$, for $\alpha>0$. Observe that $A-B_0$ is an operator of finite rank, and the rank of this operator does not exceed the double number of the ``links" that intersect one of the spheres $\{x\in\mathbb R^d:|x|=\varepsilon_k\alpha^{1/p}\}, k=1,2$. By a ``link", we mean any interval connecting two points $n,m\in\mathbb Z^d$ such that $|n-m|=1$. The number of these links is a quantity of order $O(\alpha^{(d-1)/p})$ as $\alpha\to\infty$. Denote $r(\alpha)=\text{rank}(B_0-A)$. Then $r(\alpha)=O(\alpha^{(d-1)/p})$ as $\alpha\to\infty$. 

    According to the perturbation theory, 
   \begin{eqnarray*}
       \text{rank}(E_{B_0}(\lambda_-,\lambda_+))\le r(\alpha).  \end{eqnarray*}  
    Therefore, $\text{rank}(B-B_0)\le r(\alpha)$. Thus, $\text{rank}(B-A)\le\text{rank}(B-B_0)+\text{rank}(B_0-A)\le2 r(\alpha)$. 
    
    Define $\tilde X(\lambda)=\sqrt{V}(B-\lambda I)^{-1}\sqrt{V}$ for $\lambda\in\Lambda$. Then 
    \begin{eqnarray*}
        X(\lambda)-\tilde{X}(\lambda)=\sqrt{V}(A-\lambda I)^{-1}(B-A)(B-\lambda I)^{-1}\sqrt{V},
    \end{eqnarray*} which implies that
    \begin{eqnarray*}
        \text{rank}(X(\lambda)-\tilde X(\lambda))\le\text{rank}(B-A)\le2r(\alpha).
    \end{eqnarray*} Consequently, for any $\lambda\in\Lambda$,
    \begin{eqnarray}
    \label{r(alpha)}
        |n_+(\alpha^{-1},X(\lambda))-n_+(\alpha^{-1},\tilde X(\lambda))|\le 2r(\alpha).
    \end{eqnarray} 
    
    The operator $\tilde X(\lambda)$ is the Birman-Schwinger operator with $A$ replaced by $B$. Therefore, according to the Birman-Schwinger principle, $n_+(\alpha^{-1},\tilde X(\lambda))$ is the number of eigenvalues of $B-tV$ that pass through $\lambda$ as $t$ increases from 0 to $\alpha$. Note now that the operator $B$ can be decomposed to the orthogonal sum $B=B_1\oplus B_2\oplus
    B_3$, where 
    \begin{eqnarray*}
        B_k=A_{\Omega_k(\alpha)}-(\lambda_+-\lambda_-)E_{A_{\Omega_k(\alpha)}}(\Lambda), \qquad k=1,2,3.    \end{eqnarray*} Hence,
        \begin{eqnarray*}
            (B-\lambda I)^{-1}=(B_1-\lambda I)^{-1}\oplus(B_2-\lambda I)^{-1}\oplus(B_3-\lambda I)^{-1},            
        \end{eqnarray*} which leads to the decomposition of the Birman-Schwinger operator
        \begin{eqnarray*}
            \tilde{X}(\lambda)=\tilde{X}_1(\lambda)\oplus\tilde{X}_2(\lambda)\oplus\tilde{X}_3(\lambda),\qquad \text{where }   \tilde{X}_k(\lambda)=\sqrt{V}(B_k-\lambda I)\sqrt{V}.
        \end{eqnarray*} So we finally conclude that
        \begin{eqnarray*}
            n_+(\alpha^{-1},\tilde{X}(\lambda))=n_+(\alpha^{-1},\tilde{X}_1(\lambda))+n_+(\alpha^{-1},\tilde{X}_2(\lambda))+n_+(\alpha^{-1},\tilde{X}_3(\lambda)).
        \end{eqnarray*}
        It remains to note that
        \begin{eqnarray*}
            n_+(\alpha^{-1},\tilde{X}_1(\lambda))=N_1(\lambda,\alpha),\qquad n_+(\alpha^{-1},\tilde{X}_2(\lambda))=N_2(\lambda,\alpha)            
        \end{eqnarray*} and show that
        \[
        \label{X_3}
        n_+(\alpha^{-1},\tilde{X}_3(\lambda))=0.        \] Put differently, 
\begin{eqnarray}
\label{tilde X}
    n_+(\alpha^{-1},\tilde X(\lambda))=N_1(\lambda,\alpha)+N_2(\lambda,\alpha). 
\end{eqnarray} Combining \eqref{Birman-Schwinger Principle},\eqref{r(alpha)} and \eqref{tilde X}, we get
    \begin{eqnarray*}
        N(\lambda,\alpha)-(N_1(\lambda,\alpha)+N_2(\lambda,\alpha))\le r(\alpha).
        \end{eqnarray*} 

Let us prove \eqref{X_3}. First we choose $\delta>0$ so that 
\begin{eqnarray}
\label{epsilon, delta}
    \varepsilon_2>\left(\|\Psi\|_\infty\left(\frac{1}{|\lambda-\lambda_+|}+\frac{1}{|\lambda-\lambda_-|}\right)(1+\delta)\right)^{1/p}.
\end{eqnarray}
After that, we choose $\a_0>0$ so large that
\begin{eqnarray}
\label{V}
|V(n)|\le\frac{\|\Psi\|_\infty(1+\delta)}{|n|^{p}}, \qquad\forall n\in\Omega_3(\a),\qquad \a>\a_0.
\end{eqnarray} Then it follows from \eqref{V} that
\begin{eqnarray}
\label{X_3(2)}
\|\tilde{X}_3(\lambda)\|\le\sup_{n\in\Omega_3(\a)}|V(n)|\cdot\|(B_3-\lambda I)^{-1}\|\le\frac{\|\Psi\|_\infty(1+\delta)}{\varepsilon_2^p\a}\left(\frac{1}{|\lambda-\lambda_+|}+\frac{1}{|\lambda-\lambda_-|}\right).
\end{eqnarray} Thus, by \eqref{epsilon, delta} and \eqref{X_3(2)}, $\|\tilde{X}_3(\a)\|<\a^{-1}$, which implies \eqref{X_3}.
\end{proof}

We see from this proposition that to obtain the asymptotic formula for  $N(\lambda,\alpha)$, it is enough to get asymptotic formulas for $N_1(\lambda,\alpha)$ and $N_2(\lambda,\alpha)$. Later we will prove that
\begin{equation}
    \label{1}\limsup_{\alpha\to\infty} \alpha^{-d/p}N_1(\lambda,\alpha)\leq 2\omega_d\varepsilon_1^d,
\end{equation}
where $\omega_d$ is the  volume of the unit ball in ${\mathbb R}^d$. We will also show that
\begin{eqnarray}
\label{N_2}
    \lim_{\alpha\to\infty} \alpha^{-d/p}N_2(\lambda,\alpha)= \int_{\varepsilon_1<|x|<\varepsilon_2}\left(\rho(\lambda+\Psi(\theta)|x|^{-p})-\rho(\l)\right)dx,
\end{eqnarray}
for $\varepsilon_2>(\|\Psi\|_\infty/|\lambda_+-\lambda|+\|\Psi\|_\infty/|\lambda_--\lambda|)^{1/p}$.

\section{Estimate of $N_1(\lambda,\alpha)$}

In this section, we estimate the function $N_1(\lambda,\alpha)$. The dimension of the space $l^2(\Omega_1(\alpha)\cap\mathbb Z^d)$ coincides with the number of the points in $\Omega_1(\alpha)\cap\mathbb Z^d$. On the other hand, the number of these points cannot be larger than the volume of the ball $\{x\in\mathbb R^d:|x|<\varepsilon_1\alpha^{1/p}\}$. Thus, $\dim\big( l^2(\Omega_1(\alpha)\cap\mathbb Z^d)\big)\le\text{vol}\{x\in\mathbb R^d:|x|<\varepsilon_1\alpha^{1/p}+\sqrt{d}\}=\omega_d(\varepsilon_1\alpha^{1/p}+\sqrt{d})^d$, where $\omega_d$ is the volume of the unit ball in $\mathbb R^d$. Since $N_1(\lambda,\alpha)\le2\dim\big( l^2(\Omega_1(\alpha))\cap\mathbb Z^d\big)$, we obtain \eqref{1}.

\section{Asymptotics of $N_2(\lambda,\alpha)$}

Let $\tilde\Omega=\{x\in\mathbb R^3: \varepsilon_1<|x|<\varepsilon_2\}$. In order to obtain an asymptotic formula for $N_2(\lambda,\alpha)=N(A_{\Omega_2(\alpha)}-\alpha V, \lambda)-N(A_{\Omega_2(\a)},\l)$, we divide $\tilde\Omega$ into a finite number of disjoint sets $\{Q_i\}$ given by $Q_i=(\delta\mathbb Q_i)\cap\tilde\Omega$, where $\mathbb Q_i=[0,1)^d+i, i\in\mathbb Z^d$ and $\delta>0$. Some of $Q_i$'s are cubes and some of them are not. Observe that since $\Omega_2(\alpha)\cap\mathbb Z^d$ is the disjoint union of sets $\alpha^{1/p}Q_i\cap\mathbb Z^d$, we have
\begin{eqnarray*}
    l^2(\Omega_2(\alpha)\cap\mathbb Z^d)=\oplus\,\text{-}\sum_i l^2(\alpha^{1/p}Q_i\cap\mathbb Z^d).
\end{eqnarray*} This orthogonal decomposition allows one to consider the operator
\begin{eqnarray*}
    D(\alpha)=\oplus\,\text{-}\sum_i(A_{\alpha^{1/p}Q_i}-\alpha V).
\end{eqnarray*}
Then $A_{\Omega_2(\alpha)}-D(\alpha)$ is an operator of finite rank, and the rank of this operator does not exceed the double number of the ``links" that intersect the boundaries of the sets $\alpha^{1/p}Q_i$. The number of these links is a quantity of order $O(\alpha^{(d-1)/p})$ as $\alpha\to\infty$. Therefore, 
\begin{eqnarray*}
    |N_2(\lambda,\alpha)-N(D(\alpha),\lambda)|\le\text{rank}(A_{\Omega_2(\alpha)}-D(\alpha))=O(\alpha^{(d-1)/p}),\qquad\text{ as}\quad\alpha\to\infty.
\end{eqnarray*} Since $N(D(\alpha),\lambda)=\sum_iN(A_{\alpha^{1/p}Q_i}-\alpha V, \lambda)$, we obtain
\begin{eqnarray*}
    N(A_{\Omega_2(\alpha)}-\alpha V, \lambda)=\sum_iN(A_{\alpha^{1/p}Q_i}-\alpha V, \lambda)+O(\alpha^{(d-1)/p}), \qquad\text{ as}\quad\alpha\to\infty.
\end{eqnarray*} Suppose first that the equality $V(x)=\Psi(\theta)/|x|^p$ holds (not only asymptotically, but) exactly for $|x|>1$. For each $Q_i$, define $x_i^{\max}$ and $x_i^{\min}$ to be the points for which
\begin{eqnarray*}
    \max_{x\in Q_i}V(x)=V(x_i^{\max}),\qquad \min_{x\in Q_i}V(x)=V(x_i^{\min}).
\end{eqnarray*} Then for any $n\in\alpha^{1/p}Q_i$, we have 
\begin{eqnarray}\label{doubleestimate}
    V(x_i^{\min})\le\alpha V(n)\le V(x_i^{\max}),
\end{eqnarray} where both the lower and upper bounds of $\alpha V(n)$ are independent of $\alpha$. 

By the Minimax principle, 
\begin{eqnarray*}
    N(A-\alpha V, \lambda)=\max_F\dim(F),
\end{eqnarray*} where the maximum is taken over all subspaces $F\subseteq l^2$ on which $\langle(A-\alpha V)u,u\rangle\le\lambda\|u\|^2$ for all $u\in F$. This implies that if $V_1$ and $V_2$ are two potentials defined on $\mathbb Z^d$ such that $V_1\le V_2$, then $N(A-\alpha V_1,\lambda)\le N(A-\alpha V_2,\lambda)$. Similarly, we have $N(A_{\alpha^{1/p}Q_i}-\alpha V_1,\lambda)\le N(A_{\alpha^{1/p}Q_i}-\alpha V_2,\lambda)$ if $V_1\le V_2$. 

It follows from \eqref{doubleestimate} that
\begin{eqnarray*}
    N(A_{\alpha^{1/p}Q_i}- V(x_i^{\min}),\lambda)\le N(A_{\alpha^{1/p}Q_i}-\alpha V,\lambda)\le N(A_{\alpha^{1/p}Q_i}- V(x_i^{\max}),\lambda).
\end{eqnarray*} Observe that
\begin{eqnarray*}
    N(A_{\alpha^{1/p}Q_i}-V(x_i^{\max}),\lambda)&=&N(A_{\alpha^{1/p}Q_i},\lambda+V(x_i^{\max}))\qquad\text{ and}\\ N(A_{\alpha^{1/p}Q_i}-V(x_i^{\min}),\lambda)&=&N(A_{\alpha^{1/p}Q_i},\lambda+V(x_i^{\min})),
\end{eqnarray*} thus, we have, as $\alpha\to\infty$,
\begin{eqnarray*}
    N(A_{\alpha^{1/p}Q_i},\lambda+V(x_i^{\max}))&\sim&\text{vol}(\alpha^{1/p}Q_i)\rho(\lambda+V(x_i^{\max}))=\alpha^{d/p}\text{vol}(Q_i)\rho(\lambda+V(x_i^{\max})),\\ N(A_{\alpha^{1/p}Q_i},\lambda+V(x_i^{\min}))&\sim&\text{vol}(\alpha^{1/p}Q_i)\rho(\lambda+V(x_i^{\min}))=\alpha^{d/p}\text{vol}(Q_i)\rho(\lambda+V(x_i^{\min})).
\end{eqnarray*} Consequently, the upper and lower bounds of
\begin{eqnarray*}    \lim_{\alpha\to\infty}\frac{N(A_{\Omega_2(\a)}-\a V,\l)}{\alpha^{d/p}}
\end{eqnarray*} are the Riemann sums
\begin{eqnarray*}
    &&\sum_{i=1}^k\rho(\lambda+V(x_i^{\max}))\text{vol}(Q_i)\text{ and} \\&&\sum_{i=1}^k\rho(\lambda+V(x_i^{\min}))\text{vol}(Q_i),
\end{eqnarray*}respectively. Therefore, when $\text{vol}(Q_i)\to 0$ for each $i$,
\begin{eqnarray}
\label{N=int1}
\lim_{\alpha\to\infty}\frac{N(A_{\Omega_2(\a)}-\a V,\l)}{\alpha^{d/p}}=\int_{\varepsilon_1<|x|<\varepsilon_2}\rho(\lambda+\Psi(\theta)|x|^{-p})dx.
\end{eqnarray} On the other hand, according to the definition of the density of states, 
\begin{eqnarray}
\label{N=int2}
\lim_{\alpha\to\infty}\frac{N(A_{\Omega_2(\a)},\l)}{\alpha^{d/p}}=\rho(\l)\text{vol}(\tilde{\Omega})=\int_{\varepsilon_1<|x|<\varepsilon_2}\rho(\l)dx.
\end{eqnarray} Subtracting \eqref{N=int2} from \eqref{N=int1}, we obtain \eqref{N_2}.

\section{End of the proof of Theorem~\ref{main theorem}}

It follows from \eqref{1} and \eqref{N_2} that 
\begin{eqnarray*}
    \limsup_{\alpha\to\infty}\frac{N(\lambda,\alpha)}{\alpha^{d/p}}&=&\limsup_{\alpha\to\infty}\frac{N_1(\lambda,\alpha)+N_2(\lambda,\alpha)}{\alpha^{d/p}}\\
    &\le&\omega_d\varepsilon_1^d+\int_{\varepsilon_1<|x|<\varepsilon_2}\left(\rho(\lambda+\Psi(\theta)|x|^{-p})-\rho(\l)\right)dx.
\end{eqnarray*} Taking the limits as $\varepsilon_1\to0$ and $\varepsilon_2\to\infty$, we obtain 
\begin{eqnarray}
\label{61}
\limsup_{\alpha\to\infty}\frac{N(\lambda,\alpha)}{\alpha^{d/p}}\le\int_{\mathbb R^d}\left(\rho(\lambda+\Psi(\theta)|x|^{-p})-\rho(\l)\right)dx.
\end{eqnarray} On the other hand, it also follows from \eqref{1} and \eqref{N_2} that
\begin{eqnarray*}
    \liminf_{\alpha\to\infty}\frac{N(\lambda,\alpha)}{\alpha^{d/p}}&=&\liminf_{\alpha\to\infty}\frac{N_1(\lambda,\alpha)+N_2(\lambda,\alpha)}{\alpha^{d/p}}\\
    &\ge&\int_{\varepsilon_1<|x|<\varepsilon_2}\left(\rho(\lambda+\Psi(\theta)|x|^{-p})-\rho(\l)\right)dx.
\end{eqnarray*} Taking the limits as $\varepsilon_1\to0$ and $\varepsilon_2\to\infty$, we obtain 
\begin{eqnarray}
\label{62}
\liminf_{\alpha\to\infty}\frac{N(\lambda,\alpha)}{\alpha^{d/p}}\ge\int_{\mathbb R^d}\left(\rho(\lambda+\Psi(\theta)|x|^{-p})-\rho(\l)\right)dx.
\end{eqnarray} Combining \eqref{61} and \eqref{62}, we get \eqref{Vasympt}.

\end{document}